\documentclass[12pt,a4paper]{article}

\usepackage{amsmath}
\usepackage{amsfonts}
\usepackage{amssymb}
\usepackage{amsthm}

\newtheorem{theorem}{Theorem}
\newtheorem{lemma}{Lemma}

\newtheorem*{theorem*}{Theorem}
\newtheorem*{corol*}{Corollary}
\newtheorem{corollary}{Corollary}

\author{Vaibhav Pandey, Sagar Shrivastava, and Balasubramanian Sury}
\title{A Dedekind Domain with Nontrivial Class Group}
\date{}
\begin{document}
\maketitle

\begin{abstract}
We show that the ring of real-analytic functions on the unit circle
is a Dedekind domain with class number two.
\end{abstract}

\section{Rings that engage analysts.}

\noindent Analytic properties of function spaces over the real and
the complex fields are in some ways different. This is strongly
reflected in these spaces' algebraic properties. For instance, the
ring of real-valued continuous functions on a closed interval such
as $[0,1]$ behaves similarly to the corresponding ring of
complex-valued functions; they depend only on the topology of
$[0,1]$. The ring of real-valued polynomial functions on the unit
circle can be identified with the ring of all real trigonometric
polynomials. It is not a unique factorization domain as is
demonstrated by the equation
$$\cos^2(t) = (1+ \sin(t))(1- \sin(t)).$$
In fact, the above ring is $\mathbf{R}[X,Y]/(X^2+Y^2-1)$ and the
equation $Y^2 = (1+X)(1-X)$ that holds in the quotient ring gives
two different decompositions of $Y^2$ into irreducible elements
$Y,1+X,1-X$. On the other hand, the ring
$\mathbf{C}[X,Y]/(X^2+Y^2-1) \cong \mathbf{C}[X+iY, 1/(X+iY)]$ is a
principal ideal domain. Again, the rings of convergent power series
(over either of these fields) with radius of convergence larger than
some positive real number $\rho$ is a Euclidean domain (and hence, a
principal ideal domain)\textemdash this can be seen by using the
function that counts zeroes (with multiplicity) in the disk $|z|
\leq \rho$ as a Euclidean "norm" function (see \cite{jensen}).
\vskip 3mm

\noindent In this note, we consider the rings
$C_{an}(S^1;\mathbf{R})$ of analytic functions on the unit circle
$\mathit{S}^1$ that are real-valued and the corresponding ring
$C_{an}(S^1; \mathbf{C})$ of analytic functions that are
complex-valued. We will see that the latter is a principal ideal
domain while the former is a Dedekind domain which is not a
principal ideal domain---the class group having order $2$. \vskip
5mm

\section{Maximal ideals are points.}

\noindent The proof of the fact alluded to is exactly the same as
the corresponding proof (that is well known) for the ring of
continuous functions on a closed interval.

\begin{lemma}
Maximal ideals of $C_{an}(S^1, \mathbf{R})$ and of $C_{an}(S^1,
\mathbf{C})$ are points.
\end{lemma}

\begin{proof}
This is a consequence of the compactness of $\mathit{S}^1$. Indeed,
for each point $p \in S^1$, the ideal
$$\mathfrak m_p := \{f : f(p)=0 \}$$
is maximal as the quotient is isomorphic to a field. Let us observe
that every maximal ideal $\mathfrak{m}$ is of the form $\mathfrak
m_p$ for some $p$ in $\mathit{S}^1$. If not, then we can find
functions $f_i$ in $\mathfrak{m}$ that do not vanish in a
neighborhood of $p_i$ for each $p_i$ in $\mathit{S}^1$ by
continuity. These neighborhoods cover $\mathit{S}^1$. By compactness
of $\mathit{S}^1$, finitely many of these neighborhoods cover it.
Call these $f_1, \ldots, f_n$. Then the function $\sum_{i=1}^n
\overline{f_i}f_i$ lies in $\mathfrak{m}$ and is a unit (as it does
not vanish anywhere). This is a contradiction as maximal ideals are
proper.
\end{proof}

\noindent
\begin{lemma}
The ring $C_{an}(S^1, \mathbf{C})$ of complex-valued analytic
functions on $\mathit{S}^1$ is a PID. \end{lemma}

\begin{proof}
Clearly, the maximal ideal $\mathfrak{m}_p$ is the principal ideal
generated by $z-p$. Hence, every finite product
$$\mathfrak{m}_{p_1}^{a_1} \cdots \mathfrak{m}_{p_k}^{a_k}$$
is principal. We show that every nonzero ideal is such a finite
product. Any nonzero analytic function has only finitely many zeroes
as zeroes are isolated and $S^1$ is compact. Hence, if $I$ is any
nonzero ideal, it has only finitely many common zeroes, say $p_1,
\ldots, p_k$. Let $a_i$ be the smallest positive integer such that
every element of $I$ has a zero of order at least $a_i$ at $p_i$.
Hence, for each $0 \neq f \in I$, we have $f = \bigg(\prod_{i=1}^k
(z-p_i)^{a_i} \bigg) g$ for some analytic function $g$. In other
words, $I$ is contained in $\mathfrak{m}_{p_1}^{a_1} \cdots
\mathfrak{m}_{p_k}^{a_k}$. Then
$$J := \{f/\prod_{i=1}^k (z-p_i)^{a_i} : f \in I \}$$ is an ideal.
If $J$ were a proper ideal, it would be contained in some
$\mathfrak{m}_p$. If $p \not\in \{p_1, \ldots, p_k \}$, then $I
\subset \mathfrak{m}_p$ which contradicts the fact that $p_1,
\ldots, p_k$ are the only common roots of $I$. Hence $p=p_i$ for
some $1 \leq i \leq k$. But if $f_i \in I$ has order exactly $a_i$
at $p_i$, then $f_i/\prod_{j=1}^k (z-p_j)^{a_j}$ cannot vanish at
$p_i$, a contradiction. Hence $J$ is the unit ideal and so
$$I = \mathfrak{m}_{p_1}^{a_1} \cdots \mathfrak{m}_{p_k}^{a_k}.$$
Hence $I$ is principal. So, $C_{an}(S^1, \mathbf{C})$ is a PID (and
hence a Dedekind domain).
\end{proof}
\vskip 5mm

\section{Ideals in $C_{an}(S^1;\mathbf{R})$.}

\noindent Let us recall that a real-analytic function in
$C_{an}(S^1;\mathbf{R})$ is a function such that $f \circ g_1$ and
$f \circ g_{-1}$ are analytic where $g_1(x) = e^{2i \pi x}$ on $(0,2
\pi)$ and $g_{-1}(x) = e^{2i \pi x}$ on $(-\pi, \pi)$. Recall we
observed that maximal ideals are points. \vskip 5mm

\begin{lemma}
The product of any two maximal ideals $\mathfrak{m}_{p_1},
\mathfrak{m}_{p_2}$ (including the case $p_1=p_2$), is principal.
\end{lemma}

\begin{proof}
In fact, it is easy to see that $\mathfrak{m}_{p_1}
\mathfrak{m}_{p_2}$ can be generated by the analytic function
$f_{p_1,p_2}(x) = \cos(x - \frac{(p_1+p_2)}{2})-
\cos(\frac{p_1-p_2}{2})$. To clarify this further, note that when
$p_1 \neq p_2$, the function $f_{p_1,p_2}$ has simple zeroes at
$p_1$ and $p_2$ and no other zeroes (consider the derivative). If
$p_1=p_2=p$, then the function $f_{p,p}(x) =  2
\sin^2(\frac{x-p}{2})$ has a double root at $p$ and no other roots.
In either case, it follows that any element $f \in
\mathfrak{m}_{p_1} \mathfrak{m}_{p_2}$ satisfies the property that
$\frac{f}{f_{p_1,p_2}}$ is analytic. This completes the proof.
\end{proof}

\noindent This immediately implies the following corollary. \vskip
2mm

\begin{corollary}
An ideal $I = \mathfrak m_{p_1}^{a_1} \mathfrak m_{p_2}^{a_2} \cdots
\mathfrak m_{p_n}^{a_n}$ is principal if $a_1+ \cdots + a_n$ is
even.
\end{corollary}

\begin{lemma}
An ideal $I = \mathfrak m_{p_1}^{a_1} \mathfrak m_{p_2}^{a_2} \cdots
\mathfrak m_{p_n}^{a_n}$ is not principal if $a_1+ \cdots+a_n$ is
odd.
\end{lemma}

\begin{proof}
We first show that maximal ideals in the ring
$C_{an}(S^1;\mathbf{R})$ are not principal. This is obvious because
identifying $S^1$ with $\mathbf{R}/2 \pi \mathbf{Z}$, the number of
zeroes of any analytic function on $S^1$ counted with multiplicity
is even---this is simply because of the intermediate value theorem.
Now, if $I = \mathfrak m_{p_1}^{a_1} \mathfrak m_{p_2}^{a_2} \cdots
\mathfrak m_{p_n}^{a_n}$ with $a_1 + \cdots + a_n$ odd, then $I =
\mathfrak{m}_{p_1} (g)$ by the even case. If $I= (f)$, then $f \in
g\mathfrak{m}_{p_1} \subset (g)$ so that $f/g$ is analytic. But then
$\mathfrak{m}_{p_1} = (f/g)$, which is a contradiction, as $f/g$ has
an even number of zeroes counting multiplicity, while
$\mathfrak{m}_{p_1}$ has only a common zero at $p_1$.
\end{proof}
\vskip 3mm

\noindent Finally, we have the following factorization result.
\vskip 2mm

\begin{theorem}
Every nonzero proper ideal in the ring $C_{an}(S^1;\mathbf{R})$ is
of the form $\mathfrak m_{p_1}^{a_1} \mathfrak m_{p_2}^{a_2} \ldots
\mathfrak m_{p_n}^{a_n}$ for points $p_1, \cdots, p_n$.
\end{theorem}

\noindent Before proving this theorem, we observe the following very
interesting fact. \vskip 3mm

\begin{corollary}
$C_{an}(S^1;\mathbf{R})$ is a Dedekind domain which has class number $2$.
\end{corollary}

\begin{proof}
This immediately follows from Theorem 1, Corollary 1 and Lemma 4.
\end{proof}
\vskip 3mm

\noindent {\bf Remarks on Dedekind domains and class groups.} Let us
recall briefly the role of Dedekind domains in number theory.
Dedekind domains are precisely the class of integral domains in
which the fractional ideals are invertible. The rings of algebraic
integers in finite extension fields of $\mathbb{Q}$ are natural
examples of Dedekind domains. Moreover, any PID is a Dedekind
domain. The class group of a Dedekind domain is the group of
fractional ideals modulo the principal fractional ideals. A Dedekind
domain is a unique factorization domain if and only if the class
group of fractional ideals is trivial. Many subtleties involved in
solving Diophantine equations arise from the fact that many rings of
algebraic integers arising in their study have nontrivial class
group. The Fermat equation cannot be studied by elementary algebraic
methods due to the (amazing) fact that the ring of integers in the
field generated by the $p$th roots of unity for a prime $p$ is not a
UFD for any prime $p \geq 23$. By a theorem of L. Claborn (see
\cite{claborn}), every abelian group can be realized as the class
group of some Dedekind domain; the analogous problem is open for
rings of algebraic integers. In other words, it is expected but
still unknown whether every finite abelian group can be realized as
the class group of a ring of integers in an algebraic number field.
\vskip 5mm

\noindent Finally, let us prove Theorem 1. \vskip 2mm

\begin{proof}
Consider any proper, nonzero ideal $\mathcal{I}$. Let $\left\lbrace
p_1, \ldots , p_n \right\rbrace$ be the common zeros of
$\mathcal{I}$---as we observed above, this is finite, as every
nonzero analytic function on $S^1$ has only finitely many zeroes.
For each $k \leq n$, let $a_k$ be minimal among the orders of zeroes
of elements of $\mathcal{I}$ at $p_k$. Then, it is clear that
$$\mathcal{I} \subset\mathfrak m_{p_1}^{a_1} \mathfrak m_{p_2}^{a_2}
\cdots \mathfrak m_{p_n}^{a_n}.$$ \vskip 3mm

\noindent We will show that $\mathcal{I} = \prod_{k=1}^n
\mathfrak{m}_{p_k}^{a_k}.$ \vskip 2mm

\noindent Let us first assume that $a_1+ \cdots+a_n$ is even. \\
Let $f$ be an element of $\mathcal{I}$ whose order of zero at $p_k$
is $a_k$ for $1 \leq k \leq n$. Such an $f$ exists since
$\mathcal{I}$ contains elements $f_k$ vanishing at $p_k$ with order
$a_k$, and we may consider a suitable linear combination $g_1f_1 +
\cdots + g_nf_n$. This function $f$ may have other zeroes different
from $p_k$; we wish to change $f$ such that the new element is in
$\mathcal{I}$, has zeroes of order $a_k$ at $p_k$, and has no other
zeroes. This is accomplished as follows. \\
As we observed in the beginning, every analytic function changes
signs an even number of times by the intermediate value theorem. Let
us write $0 \leq p_1<p_2< \cdots < p_n < 2 \pi$. In some of the
intervals $[p_i,p_{i+1}]$ (among $[p_1,p_2],[p_2,p_3], \ldots,
[p_n,p_1]$), the function $f$ changes sign an even number of times
and in others, it changes sign an odd number of times. The latter
happens in an even number of intervals $[p_i,p_{i+1}]$. If we select
some analytic function $g$ that has simple zeroes at some interior
point of each of these latter intervals and no other zeroes, then
the function $fg \in \mathcal{I}$ and has the property that $fg$
vanishes at each $p_i$ exactly to the order $a_i$ and has an even
number of sign changes in each interval $(p_1,p_2), (p_2,p_3),
\ldots, (p_n,p_1)$. It also changes signs at an even number of the
$p_i$'s. At these even number of $p_i$'s, there is an analytic
function $h$ with simple zeroes and no other roots. We may multiply
the analytic function $h$ by an element $\phi \in \mathcal{I}$ that
has no zeroes in any of the open intervals $(p_i,p_{i+1})$ (we may
square and assume the value of $\phi$ is positive in each of these
open intervals). By changing the sign of $h$ if necessary, we may
assume it has the same sign as $f$ around each $p_i$. Then $h \phi
\in \mathcal{I}$ has simple zeroes at the $p_i$'s, no other zeroes,
and has the sign of $f$ in each open interval $(p_i,p_{i+1})$. Since
continuous (hence analytic) functions are bounded on a compact set,
therefore, for a large constant $c$, the function $fg + ch \phi$ is
in $\mathcal{I}$ and has zeroes of order $a_i$ at $p_i$ and no other
zeroes. Hence $\mathcal{I} \supseteq \prod_{i=1}^n
\mathfrak{m}_{p_i}^{a_i}$ which shows that these ideals are equal.
\vskip 3mm

\noindent Now, assume that $a_1+ \cdots + a_n$ is odd. Let $f \in
\mathfrak m_{p_1}^{a_1} \mathfrak m_{p_2}^{a_2} \cdots \mathfrak
m_{p_n}^{a_n}$. Since it must have an even number of zeros (counting
multiplicity), it must have a zero $q \not\in \left\lbrace p_1,
\ldots , p_n \right\rbrace$ or it has a zero of order greater than
$a_k $ at some $p_k$ (in which case we put $q=p_k$). Let
$\mathcal{J} = \left\lbrace g \in \mathcal{I} \vert g(q) =0
\right\rbrace$ (in case $q=p_k$, we take $g$ to have zeroes at $p_k$
with multiplicity greater than $a_k$). By the even case treated
already, $\mathcal{J} = \mathfrak m_{p_1}^{a_1} \mathfrak
m_{p_2}^{a_2} \cdots \mathfrak m_{p_n}^{a_n}\mathfrak m_{q}$. Thus
$f$ belongs to the right-hand side. Thus $f \in \mathcal{J}\subset
\mathcal{I}$. This completes the proof.
\end{proof}
\vskip 5mm

\noindent We end with a remark which is relevant to the fact that
the ring of real analytic functions on $S^1$ has class number $2$.
L. Carlitz proved (in \cite{carlitz}) that Dedekind domains with
class number $2$ are {\it half-factorial domains}; viz., different
irreducible factorizations of elements must have the same length.
Finally, we mention that we are interested in generalizations of the
above result to compact manifolds. \vskip 5mm

\begin{acknowledgment}
{Acknowledgments.} The second author was a project student at the
Tata Institute of Fundamental Research, Mumbai, India when this note
was written. The authors would like to thank the two referees for
suggestions to rewrite some parts more clearly. Interestingly, one
of the referees brought the authors' attention to a link
https://math. stackexchange.com/questions/
154605/ring-of-analytic-functions-on-the-circle which contains a
discussion of the same example studied in this note.
\end{acknowledgment}

\vskip 5mm

\begin{affil}
University of Utah, 1015, Barbara Place, Apt-3, Salt Lake City -
84102, USA. \\Email: pandey@math.utah.edu
\end{affil}

\begin{affil}
School of Mathematics, Tata Institute of Fundamental Research, Homi
Bhabha Road, Mumbai, 400005, India. \\Email: sagars@math.tifr.res.in
\end{affil}

\begin{affil}
Statistics \& Mathematics Unit, Indian Statistical Institute, 8th
Mile Mysore Road, Bangalore 560059, India. \\ Email:
surybang@gmail.com
\end{affil}

\end{document}